\documentclass[10pt]{amsart}
\usepackage{amsmath,amssymb,xy,array}
\usepackage[T1]{fontenc}
\usepackage{amsfonts}
\usepackage{pgfplots}
\pgfplotsset{compat=1.15}
\usepackage{mathrsfs}
\usetikzlibrary{arrows}
\definecolor{ttttff}{rgb}{0.2,0.2,1}
\definecolor{qqqqcc}{rgb}{0,0,0.8}
\definecolor{wrwrwr}{rgb}{0.3803921568627451,0.3803921568627451,0.3803921568627451}
\usepackage{amsmath}
\usepackage{amssymb}
\usepackage{amsthm}
\usepackage[cp850]{inputenc}
\usepackage[makeroom]{cancel}
\usepackage{hyperref}
\usepackage{graphicx}
\usepackage{fancyhdr}
\usepackage{tikz}
\usepackage{longtable}
\usepackage{geometry}
\usepackage{textcomp}
\usetikzlibrary{trees}
\newcounter{myenumi}
\renewcommand{\themyenumi}{(\roman{myenumi})}
\newenvironment{myenumerate}{%
	\setlength{\parindent}{0pt}
	\setcounter{myenumi}{0}
	\bigskip
	\renewcommand{\item}{
		\par
		\refstepcounter{myenumi}
		\makebox[2.5em][l]{\themyenumi}
	}
}{
	\par
	\bigskip
	\noindent
	\ignorespacesafterend
}

\newlength{\defbaselineskip}
\newcommand\restr[2]{{
		\left.\kern-\nulldelimiterspace 
		#1 
		\vphantom{\big|} 
		\right|_{#2} 
}}

\makeatletter
\newcommand{\tpitchfork}{%
	\vbox{
		\baselineskip\z@skip
		\lineskip-.52ex
		\lineskiplimit\maxdimen
		\m@th
		\ialign{##\crcr\hidewidth\smash{$-$}\hidewidth\crcr$\pitchfork$\crcr}
	}%
}

\providecommand{\U}[1]{\protect\rule{.1in}{.1in}}
\providecommand{\U}[1]{\protect\rule{.1in}{.1in}}
\providecommand{\U}[1]{\protect\rule{.1in}{.1in}}
\providecommand{\U}[1]{\protect\rule{.1in}{.1in}}
\providecommand{\U}[1]{\protect\rule{.1in}{.1in}}
\input{xy}
\xyoption{all}
\usepackage{amsfonts}
\usepackage{amssymb,amsthm}
\usepackage{verbatim}
\usepackage{hyperref} 
\usepackage[T1]{fontenc}
\usepackage{amsmath}
\usepackage{hyperref}
\usepackage{enumerate}
\usepackage{tikz}
\usepackage{color}

\newcommand{\QED}{\ifhmode\unskip\nobreak\fi\quad {\rm Q.E.D.}} 

\newtheorem{thm}{Theorem}[section]

\newtheorem{remark}[thm]{Remark}

\newtheorem{Lemma}[thm]{Lemma}
\newtheorem{Proposition}[thm]{Proposition}
\newtheorem{Problem}{Problem}
\newtheorem{Corollary}[thm]{Corollary}

\theoremstyle{definition}

\newtheorem{Definition}[thm]{Definition}

\newtheorem{Example}[thm]{Example}

\hypersetup{pdfpagemode=UseNone}
\hypersetup{pdfstartview=FitH}

\begin{document}
	\title{Monodromies of Projective Structures on Surface of Finite-type}
	-
	\author[Genyle Nascimento]{Genyle Nascimento}
	\address{\sc
		Universidade Federal Fluminense\\
		Campus Gragoat\'a, Rua Alexandre Moura 8 - S\~ao Domingos\\
		24210-200, Niter\'oi, Rio de Janeiro\\ Brazil}
	\email{	genyle.nascimento@gmail.com
	}
	
	\date{\today}
	
	\maketitle

	\setcounter{tocdepth}{1}

	\begin{abstract}
		We characterize the monodromies of projective structures with fuchsian-type singularities. Namely, any representation from the fundamental group of a Riemann surface of finite-type in $PSL_2(\mathbb{C})$ can be represented as the holonomy of branched projective structure with fuchsian-type singularities over the cusps. We made a geometrical/topological study of all local conical projective structures whose Schwarzian derivative admits a simple pole at the cusp. Finally, we explore isomonodromic deformations of such projective structures and the problem of minimizing angles. 
	\end{abstract}
	
	\tableofcontents
	%
	%
	\section{Introduction}
	This theory has its roots in the study of automorphic functions and differential equations by Klein \cite[Part 1]{K}, Poincar\'e \cite{P}, Riemann \cite{R}, and others in the late nineteenth century (see Hejhal's works \cite{He1}, \cite{He} for further historical discussion and references).
	
	 A complex projective structure on an oriented surface is a distinguished system of local coordinates modeled in $\mathbb{CP}^1$ in such a way that transition maps extend to homographies, i.e., lie in $PSL_2(\mathbb{C})$. Branched projective structures on closed orientable surfaces are given by atlases where local charts are finite branched coverings and transition maps lie in $PSL_2(\mathbb{C})$. 
	
	We know that if $S$ is a surface with a projective structure modeled in the projective space $\mathbb{CP}^1$, then there is a pair $(dev,\rho)$ (unique up to the action of automorphisms of $\mathbb{CP}^1$), where $ dev: \tilde{S} \rightarrow \mathbb{CP}^1$, defined on the universal covering of $S$ is a projective immersion equivariant with respect to the homomorphism $\rho:\pi_1(S) \rightarrow Aut(\mathbb{CP}^1)$. Two projective structures are equivalent if the developing maps $dev$ differ by a homography.

	We use a point of view for projective structure that  is a description as in which case the monodromy is conjugated in Goldman's thesis (\cite{GoT}), that associated a projective structure to a triple $(\pi, \mathcal{F}, \sigma)$ where let $ \rho: \pi_1 (S) \rightarrow Aut (\mathbb {CP}^1) $ be a representation, there exists a natural bijection between equivalence classes of complex projective structures with monodromy $ \rho $ in $ S $ and sections $\sigma$ of the $\mathbb{CP}^1$-bundle $ S \times_\rho \mathbb{CP}^1 $, suspension of $\rho$, that are transversal to the foliation $\mathcal{F}$ obtained by quotienting the horizontal foliation of $ \tilde{S} \times \mathbb{CP}^1 $.

A natural question about complex projective structures and their monodromy representations is to describe which representations can be realized as monodromy of a projective structure. In the case of closed surfaces of genus $g\geq2$, Gallo-Kapovich-Marden \cite[2000]{GKM} showed that non-elementary representations are monodromies of projective structures with at most one branch point.

They do not prescribe the complex structure in advance, rather it
is determined as part of the solution. The need to introduce a branch point in the Theorem is however reminiscent of the need for "apparent
singularities" in the theory about linear ordinary
differential equations on Riemann surfaces introduced by Poincar\'e.

Also in \cite[2000]{GKM}, they listed some open problems that we treat here about complex projective structures: 

	\begin{Problem}
		\textit{Prove and/or explore the existence and non-uniqueness of complex projective structures with given monodromy in punctured surfaces.}
	\end{Problem} 
\begin{Problem}\label{minimizing}
	\textit{Make precise and optimize the connection between branching divisors and monodromy. Namely, compute the function $d:Hom(\pi_1(S),PSL_2(\mathbb{C}))\rightarrow \mathbb{Z}$,
		where $d(\rho)$ is the smallest integer for which there exists a branched complex projective
		structure with branching divisor of degree d and monodromy $\rho$.}
\end{Problem}



	The problem of building complex projective structures on surfaces of finite-type had already been explored by Poincar\'e through his studies in solving linear differential equations to come at the Uniformization Theorem even though this was not his initial goal.

	We define a singularity of \textit{Fuchsian-type} as a point such that around it there is a map that, up to local holomorphic coordinate change, is given by $ z^\alpha $, $ \alpha \in \mathbb{C}^* $, or $ \log z + \frac{1}{z^n} $, $ n \in \mathbb{N} $. We define a \textit{singular projective structure of Fuchsian-type} in $S$ as projective structures where a finite number of the singularities of this type are allowed. These singularities are the same considered by Fuchs in his studies about differential equations (\cite{SG}).
	

We know that every projective structure on a surface has a subjacent complex structure. If we consider a surface of finite-type $S^*:=S\setminus P$, where $S$ is a compact Riemann surface and $P$ is a finite subset $\{p_1,\ldots,p_k\}$ of $S$, the complex structure extends in a unique way to $S$.
	
To obtain a result, analogous to the Gallo-Kapovich-Marden's Theorem to Riemann surfaces of finite-type, we will prove:
	
	\begin{thm}\label{teoremaexistencia}
		Let $S$ be a compact Riemann surface of any genus and $\{p_1, p_2, \ldots, p_k \} \subset S $ a finite subset with $S \setminus S^*=\{p_1, \ldots, p_k \} $. Given a representation $\rho: \pi_1(S^*) \rightarrow PSL_2(\mathbb{C})$ there exists a singular projective structure of the Fuchsian-type in $S$ with monodromy $\rho$.
	\end{thm}
	
	We prescribe the monodromy, but differently of Gallo-Kapovich-Marden, we prescribe a complex structure before building the projective structure. We do not control the local models, i.e., they cannot be optimal, it can have a finite number of singularities with trivial local monodromy to exist outside the cusps $S \setminus S^*$, i.e., branch points, and angle excesses at the cusps.
	
	The proof uses techniques from algebraic geometry precisely that ruled surfaces have sections and ideas developed by Loray and Marín \cite{LM}. Let $ \pi: P \rightarrow S $ be a $ \mathbb{P}^1$-bundle over $S$ associated to the monodromy representation $ \rho: \pi_1 (S^*) \rightarrow PSL_2 (\mathbb{C})$ of a singular projective structure of Fuchsian-type in $S$. This $\mathbb{P}^1$-bundle is equipped with a Riccati foliation $ \mathcal{F}_\rho$ (see Brunella \cite{B}) with the same monodromy obtained by compactification the suspension of the representation $\rho$. The developing map of the projective structure defined in $ S $ defines a non-trivial holomorphic section $ \sigma$ of $\mathbb{P}^1$-bundle $ \pi$ and non-invariant by $\mathcal{F}_\rho$. After, we will study the relation between isomonodromic projective structures and flipping of a fiber.
	
	This result goes back to a work by Loray and Pereira \cite{LP} that restricted to projective surfaces, it is possible to build transversely projective foliations 
	with prescribed monodromy. The approach is similar to ours, although they use other tools such as Deligne's work on the Riemann-Hilbert problem is used to build a meromorphic plane connection in a rank 2 vector bundle whose projectivization gives the $\mathbb{CP}^1$-bundle and build a meromorphic section generally transversal to the foliation using also fiber bundle's theory.

In the presence of branching, or excess we show that the solution, using the surgery of moving branch points, given in Theorem \ref{teoremaexistencia} is not rigid: we can isomonodromically deform the structure if we have the models $\log z+\frac{1}{z^{n}}$, $n\geq 2$ e $z^{\alpha}$, $\Re \alpha>1$ .
	 
We generalize this surgery when one of the singularities involved is of Fuchsian-type using the topological/geometrical description of projective charts around the singularities   and we will show the inverse surgery of moving branch points with one of the singularities has non-parabolic monodromy and excess angles.
	
We extended the notion of branching order at singular points of Fuchsian-type, in fact, it will follow from Theorem \ref{teoremaexistencia} that around each singular point $ p $ of the projective structure of Fuchsian-type $\sigma$ with given monodromy $\rho: \pi_1(S^*) \rightarrow PSL_2(\mathbb {C})$ and  the projective charts are defined by $z^{\alpha + n_p} $, $ 0 <\Re \alpha \leq1 $ or $ \log z + \frac{1}{z^{n_p}} $. We define as $ n_p \in \mathbb{Z} $ the branching order  at each singular point $p$ and the sum $e(\sigma) = \sum_{p \in S} n_p $ as branching order of the projective structure $\sigma $. 

We obtain the following result for representations that monodromies of projective structures of Fuchsian-type not of minimum branching order:
	
	\begin{thm} \label{structcomexcessos}
		Let $\rho:\pi_1(S^*) \rightarrow PSL_2(\mathbb{C})$ be a representation. A projective structure of Fuchsian-type with monodromy $\rho$ has an odd branching order if and only if:
		\begin{enumerate}
			\item $w_2(P)$ is even and the number of cusps with non-trivial local monodromy is odd; or
			\item $w_2(P)$ is odd and the number of cusps with non-trivial local monodromy is even.
		\end{enumerate}
	where $w_2(P)$ is the 2nd Stiefel-Whitney class of the bundle  $\pi:P\rightarrow S$.
	\end{thm}

Minimizing the branching order of a singular projective structure of Fuchsian-type in a surface $S$ with monodromy $\rho$ is equivalent minimizing the index $tang(\mathcal{F}_\rho, \sigma(S))$ (see Brunella \cite{B}[Section 2.2]) of a Riccati foliation and a section $\sigma$ of the $\mathbb{CP}^1$-bundle with a specific compactification of $\mathbb{CP}^1$-bundle that defines a projective structure.

We define $d(\rho) = \min \{e(\sigma): \sigma$ is a projective structure of Fuchsian-type with monodromy $\rho \}$. For the Theorem cases, we will necessarily have $ d (\rho) \geq 1 $. Given a representation $\rho:\pi_1(S^*) \rightarrow PSL_2(\mathbb{C})$, we can rewrite the Problem \ref{minimizing} as what is the minimum branching order of a projective structure of Fuchsian-type with this monodromy representation? For the Theorem \ref{structcomexcessos} cases, do we have $d(\rho) = 1 $? And in other cases do we have $d(\rho)=0$?

Independently, Gupta \cite{Gu} has also shown a version to the Gallo-Kapovich-Marden's Theorem with an analogous statement by using techniques
from hyperbolic geometry in dimension 3. One of the biggest differences with our work is a restrictive hypothesis in monodromy representations that will not provide branch points outside the cusps.

Calsamiglia, Deroin, and Francaviglia in \cite{CDF} proved that two-branched projective structures on compact surfaces with the same quasi-Fuchsian holonomy and same branching order are related by moving branch points. It would be interesting to do the same for the case of projective structures with Fuchsian-type singularities. As in \cite{CDF}, we can use the surgery debubbling for reduce the branching order.




	\subsection*{Acknowledgments}This paper is the result of the author's PhD thesis that to thank the advisor G. Calsamiglia for introduce this subject and for help and orientation. The author also acknowledge
	the financial support from Coordenação de Aperfeiçoamento de Pessoal de Nível
	Superior - Brasil (CAPES) - Finance Code 001 and  CAPES-Mathamsud.
	I am grateful to the following institutions for the very nice working conditions
	provided: École Normale Supérieure de Paris and Universidade Federal Fluminense (UFF). 
	I am thankful to B. Deroin, T. Fassarella and A. Muniz for their comments. 
	
\section{Preliminaries}

\subsection{Projective Structures}On a surface S, a \textit{complex projective structure} is defined by an atlas $\{U_i, f_i\}$ of homeomorphisms $ f_i: U_i \rightarrow V_i $, $ V_i \subset  \mathbb{CP}^1$, where the transition maps $ f_i = \phi_ {ij} \circ f_j $ are restrictions of Möbius transformations $\phi_{ij}\in PGL_2(\mathbb{C})$.

We can define of alternative way: let $ \tilde{S} \rightarrow(S,z_0)$ be a universal covering of $S$ based on $z_0 $. A pair $(dev, \rho)$ where $dev: \tilde {S} \rightarrow \mathbb{CP}^1$ is a local homeomorphism equivariant to respect the monodromy representation $\rho:\pi_1(S,z_0) \rightarrow PGL_2(\mathbb{C})$ defines a complex projective structure on $S$. Two projective structures on $ S $ where the developing maps differ by homography are equivalent.
	
A Riemann surface $S^*$ is of  \textit{finite-type} if it is biholomorphic to $S^*:=S\setminus P$, where $S$ is a  compact Riemann surface and $P$ is a finite subset $\{p_1,\ldots,p_k\}$ of $S$, we call $p_i$ of cusps. 

We define a \textit{singular projective structure} in $S$ as a complex projective structure in $S^* = S\setminus \{p_1,\ldots, p_k \}$, where $ \{p_1,\ldots,p_k\} \subset S$ is a finite subset of $S$ and each $ p_i $ is called the singularity of the structure.
	
	The restriction of a projective structure to an open subset $ U \subset S^*$ produces a projective structure in $U $ and we can consider the structures as a germ in the local ring and consider equivalence of germs of projective structures around their singularities.
	
	The monodromy of a singularity is the monodromy of the restriction of the projective structure to a disk around it.
	
	\begin{Example}
		In $S^*=\mathbb{CP}^1\setminus\{0,\infty\}$ with non-trivial monodromy, we can build projective charts as the branches of the multivalued map $z^\alpha$ with $\alpha \in \mathbb{C} \setminus \mathbb{Z} $ fixed and monodromy around the cusps conjugate to $w \mapsto e^{2\pi i\alpha}w$, as well as, the branches of $\log z + \frac{1}{z^n}$ will also define a singular projective structure with monodromy $ w \mapsto w + 2 \pi i $.
	\end{Example}
	
	We define a singularity of \textit{Fuchsian-type} as a point such that around it there is a map that, up to local holomorphic coordinate change, is given by $ z^\alpha $, $ \alpha \in \mathbb{C}^* $, or $ \log z + \frac{1}{z^n} $, $ n \in \mathbb{N} $. We define a \textit{singular projective structure of Fuchsian-type} in $S$ as projective structures where only  singularities of this type are allowed.
	
	We remark that singularities of Fuchsian-type with trivial monodromy have a simple topological description that comes from branched coverings. 	
	In particular, a branch point is a singularity of Fuchsian-type.

In an analytic approach, a projective structure is represented by a quadratic differential on a Riemann surface, which is extracted from the Schwarzian derivative.

Considering differential equations with poles corresponding to singular projective structures, works by Fuchs and later by Schwarz give meaning to the nomenclatures used by Poincar\'e such as ``Fuchsian functions'' and ``Fuchsian groups''. For equivalence between complex projective structure (compatible with the complex structure) in S and second-order linear differential equation and more details can be seen in \cite[Chapitres 8,9]{SG}.

We say that a reduced linear differential equation with a $h$ meromorphic coefficient
	\begin{equation}
		\label{eqfuchsiana}
		\frac{d^2u}{dz^2} + hu = 0
	\end{equation}
	is Fuchsian on $z=z_0$ if $h$ has a maximum of a double pole in $z_0$.
	
	The chart $w$ of the projective structure around $z_0$ is the quotient $w =\frac{u_1}{u_2}$ of two independent solutions of the equation (\ref{eqfuchsiana}) around $z_0$, or better, as a solution of Schwarzian equation
	\begin{equation}\label{scharzianequation}
		S_z(w): = \Bigg \lbrace \left (\frac {w'' z)} {w'(z)} \right)' - \frac{1}{2} \left (\frac{w ''(z)}{w'(z)} \right)^2 \Bigg \rbrace = 2h,
	\end{equation}
	where $ h $ is the coefficient of the equation (\ref{eqfuchsiana}). Then, we will say that the projective structure around $z_0 $ has a Fuchsian-type singularity in $z_0$. A meromorphic quadratic differential defined by the Schwarzian derivative of projective charts has the form
	\begin{equation} \label{difquadmeromorfa}
		\Bigg \lbrace \frac{1-\alpha^2}{2z^2}+\sum_{n\geq-1} b_nz^n\Bigg \rbrace dz^2,
	\end{equation}
	in local coordinates around each singularity of Fuchsian-type, $\alpha, b_n \in \mathbb{C}$.
	Conversely, the quotient of two linearly independent solutions of Schwarzian equation (\ref{scharzianequation}) defines a projective chart of Fuchsian-type.
	
	Fuchs-Schwarz \cite[Théorème IX.1.1.]{SG} resolves the Schwarzian equation in the neighborhood of a double pole which the quotient of solutions are, in local coordinate around the pole, $y^\alpha$, $ \alpha \in \mathbb{C}^*$ and $\frac{1}{y^\alpha} + \log y$ if $\alpha \in \mathbb{N}$.

	About the projective structure of Fuchsian-type over the three-punctured sphere, we can see as the quotient of solutions of a Schwarzian equation explicitly.
	
	\begin{thm} \label{exemploesferamenos3}(\cite{SG})Given $\alpha_0,\alpha_1,\alpha_{\infty}\in\mathbb{C}\setminus \mathbb{Z}$, the Schwarzian given by 
		\begin{equation}\label{eqschwarztriang}
			\Bigg\lbrace\frac{1-\alpha_0^2}{2z^2}+\frac{1-\alpha_1^2}{2(z-1)^2}-\frac{\alpha_0^2+\alpha_1^2-\alpha_{\infty}^2-1}{2z(z-1)}\Bigg\rbrace dz^2
		\end{equation}
		defines the only projective structure in $\mathbb{CP}^1\setminus\{0,1,\infty\}$ with charts projectively equivalent to $z\mapsto z^{\alpha_i}$, $i=0,1,\infty$, at the cusps. If $\alpha_0,\alpha_1,\alpha_{\infty}\in \mathbb{Z}$, the Laurent series expansion around $z_i\in\{0,1,\infty\}$ of the Schwarzian given by
		$$\Bigg\lbrace\frac{1-\alpha_i^2}{2(z-z_i)^2}+\sum_{n\geq-1}a_n^{(i)}(z-z_i)^n\Bigg\rbrace dz^2$$
		is the meromorphic quadratic differential of a branched projective structure if and only if $a_n^{(i)}$ satisfies the indicial equation 
		$A_{\alpha_i}(a_{-1}^{(i)},\ldots,a_{\alpha_i-1}^{(i)})=0$ where $A_{\alpha_i}$ is a polynomial with coefficients in $\mathbb{C}$. Otherwise, the charts around the cusps $z_i$ is projectively equivalent to $z\mapsto \log z+\frac{1}{z^{\alpha_i}}$. 
	\end{thm}
	
	Since any three points in $\mathbb{CP}^1$ can be taken to $0,1$ and $\infty$ by a Möbius transformation, then the projective structures with three singularities on $\mathbb{CP}^1$ are completely determined by their indexes. 
	
	Let $[\gamma_0],[\gamma_1],[\gamma_\infty]\in\pi_1(\mathbb{CP}^1\setminus\{0,1,\infty\})$ be loops around each $i$, for $i=0,1,\infty$, which have the same base point, satisfy $[\gamma_0]\cdot[\gamma_1]\cdot[\gamma_\infty]=Id$. Since this relation, the monodromy representation  $\rho:\pi_1(\mathbb{CP}^1\setminus\{0,1,\infty\})\rightarrow PSL_2(\mathbb{C})$ must satisfy
	\begin{equation}\label{relacaomonodromialocal}
		\rho([\gamma_0])\cdot\rho([\gamma_1])\cdot\rho([\gamma_\infty])=Id.
	\end{equation}
	The transformation $\rho([\gamma_i])$ is a local monodromy around $i=0,1,\infty$ and conjugate in $PSL_2(\mathbb{C})$, in the non-parabolic case, to $w\mapsto e^{2\pi i\alpha_i}w$, $\alpha_i\in\mathbb{C}$. At the cusps with parabolic monodromy is conjugate to  $w\mapsto w+2\pi i$. The relation (\ref{relacaomonodromialocal}) is equivalent to  $\alpha_0+\alpha_1+\alpha_\infty\in\mathbb{Z}$ and the representation $\rho$ induces a complex projective structure in $\mathbb{CP}^1\setminus\{0,1,\infty\}$.

	\subsection{Compactification of $\mathbb{CP}^1$-bundles}
	Different from Gallo-Kapovich-Marden, we prescribe a complex structure before building the projective structure. It is allowed a finite number of singularities with  trivial local monodromy to exist outside the cusps $S \setminus S^*$, i.e., branch points. As it was done on Goldman's thesis in \cite{GoT}, we will build these projective structures through sections of the $\mathbb{CP}^1$-bundle obtained from the suspension of the given representation.
	
	We denote by $S\times_{\rho}\mathbb{CP}^1$ the suspension of a representation $\rho:\pi_1(S)\rightarrow PSL_2(\mathbb{C})$, the construction can be found in \cite[Chapter 5]{CLN}.
	
	Let $S^*=S\setminus\{p_1,\ldots,p_k\}$ be a Riemann surface of finite-type and $\rho: \pi_1(S^*)\rightarrow PSL_2(\mathbb{C})$ a representation, we can compactify the suspension $S^*\times_{\rho}\mathbb{CP}^1$ as a fiber bundle over $S$ provided with a Riccati (possibly singular) foliation\footnote{A foliation $\mathcal{F}$ on a compact connected surface $X$ is called \textit{Riccati foliation} if there exists a $\mathbb{CP}^1$-bundle $\pi:X\rightarrow B$ (possibly with singular fibers) whose generic fiber is tranverse to $\mathcal{F}$.
	}, where the fibers over the cusps are invariant curves and each contains one or two singularities of foliation.
	
	There are local models, can be found at Brunella \cite{B}, that allow us to compactify the suspension and its foliation over the cusps. 
	
	Before that, we need a previous result that allows us to paste the local models around the cusps to the suspension.
	
	Let $ \mathbb{D} = \{z \in \mathbb{C} \ | \ | z | <1 \} $ the unit disk centered on the source and $ \mathbb{D}^* =\mathbb{D} \setminus \{0 \} $. We will denote by $(\mathbb{D}^*, \mathcal {F}, \pi)$ a Riccati foliation $ \mathcal{F} $ defined in $ \mathbb{D}^* \times \mathbb{CP}^1 $ with $ \pi $ a $\mathbb{CP}^1$-bundle transversal to the foliation  $ \mathcal{F}$.
		
		\begin{Proposition} \label{folriccatilocbihol}
			Let $ (\mathbb{D}^*, \mathcal{F}_1, \pi_1) $ and $ (\mathbb{D}^*, \mathcal{F}_2, \pi_2) $ be Riccati foliations. There is a biholomorphism $ \phi: \mathbb{D}^* \times  \mathbb{CP}^1 \rightarrow \mathbb{D}^* \times \mathbb{CP}^1 $ that takes leaves of $ \mathcal{F}_1 $ to leaves $\mathcal{F}_2 $  and such that $ \pi_1 $ and $ \pi_2 $ are equivalent bundles if and only if the representations of holonomy are analytically conjugate.
			\end{Proposition}
		This is just a  modified statement of Theorem 2 \cite{CLN} p. 98. For a complete proof with the details and modifications can be found in \cite[Proposição 2.1]{NS} (see also note after Theorem 2 \cite{CLN} p. 99).

	\begin{Lemma}
		\label{compactidefibrados}
		Every suspension $S^*\times_\rho \mathbb{CP}^1$ admits a compactification $\pi: \overline{S^*\times_\rho\mathbb{CP}^1}\rightarrow S$, $\mathbb{CP}^1$-bundle over $S$, provided with a Riccati  foliation $\mathcal {F}_\rho$ with invariant fibers over the cusps with non-trivial monodromy.
	\end{Lemma}
	
	\begin{proof}There is a regular "horizontal" foliation on $S^*\times_\rho \mathbb{CP}^1$. Let $D_i$ be a disk (image of a disk of complex plan by a chart of complex structure of $S$) on $S$ around of $p_i$, we have that the foliation over $D_i\setminus \{p_i\}$ is determined by $\rho(\partial D_i)\in PSL_2(\mathbb{C})$. We choose a biholomorphism that maps $p_i$ to $0$ and $D_i$ to $\mathbb{D}$.
		
		We can choose on $\mathbb{D}\times\mathbb{CP}^1$ a singular Riccati foliation with any prescribed monodromy, where in coordinates   $(z,w)\in\mathbb{D}\times\mathbb{CP}^1$ of   fiber bundle trivialization around of invariant fiber, the foliation will be generated by a meromorphic 1-form defined in $\mathbb{D}\times\mathbb{CP}^1$ rational in the variable $w$ (or, dually generated by vector fields). For each monodromy, we will choose the following models:
		\begin{enumerate}
			\item  In the case of non-parabolic monodromy, conjugated to $w\mapsto e^{2\pi i \alpha_i}w$,  the vector field $z\dfrac{\partial}{\partial z}+\alpha_iw\dfrac{\partial}{\partial w}$ and the 1-form $\alpha_iwdz-zdw=0$, $\alpha_i\in\mathbb{C}$, or
			\item  In the case of parabolic  monodromy, conjugated to $w\mapsto w+1$, the vector field $z\dfrac{\partial}{\partial z}+\dfrac{\partial}{\partial w}$ and the 1-form $dz-zdw=0$, or
			\item In the case of trivial monodromy, conjugated to the identity, the vector field $z\dfrac{\partial}{\partial z}+mw\dfrac{\partial}{\partial w}$ and the 1-form $mwdz-zdw=0$, for some $m\in\mathbb{N}$.
		\end{enumerate}
		
		Then, locally the monodromies are the same, so by the Proposition \ref{folriccatilocbihol} the foliations over $D_i\setminus \{p_i\}$ are biholomorphic, so we can glue and obtain a singular Riccati foliation in all $\overline{S^*\times_\rho\mathbb{CP}^1}$ where over the cusps with non-trivial monodromy has invariant fibers in $\{z=0\}$ and one or two singularities. 
		
		In the non-parabolic case, the singularities of  foliation are $(0,0)$ and $(0,\infty)$ with separatrix $\{w=0\}$ and $\{w=\infty\}$, and in the parabolic case has a saddle-node singularity in $(0,\infty)$ and a separatrix $\{w=\infty\}$.
	\end{proof}
	
	\begin{remark}In that compactification, the fibers over the cusps with non-trivial monodromy are always invariant by $\mathcal{F}_\rho$. In the cusps of trivial monodromy, only in the case $m=0$ in (3) we'd have a compactification given in the neighborhood these points by product foliation without invariant fibers and singularities.
	\end{remark}
	
\subsection{Flippings and existence of holomorphic sections}

 We will show that there is a holomorphic section generically transversal to the foliation for the bundle obtained from the Lemma \ref{compactidefibrados}. This is necessary for describing the singularities of the projective structures obtained through holomorphic sections of $ \overline{S^* \times_\rho \mathbb{CP}^1} $,
	
	Recall that a $\mathbb{CP}^1$-bundle, suspension of a representation $\rho:\pi_1(S^*)\rightarrow PSL_2 (\mathbb {C})$, has an invariant holomorphic section if and only if $\rho$ has fixed points. Each fixed point determines an invariant holomorphic section transporting the fixed point through the holonomy of the foliation. Therefore, we have at most two fixed points for non-trivial representations and therefore we will have at most two invariant sections.
	
	We say that a representation $ \rho: \pi_1(S) \rightarrow PSL_2(\mathbb{C})$ is elementary if the action of $Im(\rho)$ on $\mathbb{H}^3 $ fixes one point or two in $\mathbb{H}^3\cup\partial\mathbb{H}^3 $, otherwise, we call it non-elementary. If the representation is non-elementary, the $\mathbb{CP}^1$-bundle will not have an invariant section.
	

	We will show that if the monodromy of a Riccati foliation in a $\mathbb{CP}^1$-bundle over a Riemann surface $S$ is non-trivial, then the fiber bundle has at least three holomorphic sections to guarantee the existence of at least one non-invariant. We use a result by Tsen (\cite{BPV} p. 140) that affirms that $\mathbb{CP}^1$-bundle $\overline{S^* \times_\rho \mathbb{CP}^1}$ has a holomorphic section.

	The monodromy representation of a Riccati foliation gives a complete description of the foliation  module birational isomorphisms, according to \cite[Chapter 4]{B}.
	
	When the monodromy representation is non-parabolic (including the trivial case) the foliation around an invariant fiber is $(\alpha+n)wdz-zdw=0$ and $\alpha wdz-zdw=0$, $\alpha\in \mathbb{C}$, $n\in\mathbb{Z}$, and these foliations are related to each other by flipping (or elementary transformation) of that fiber, i.e., related through a sequence of blowings up at the singular points and contractions of invariant fibers. Flipping the fiber does not change the monodromy $w\mapsto e^{2\pi i \alpha}w$ around of fiber. Similarly, when the monodromy is parabolic the foliations are $dz-zdw=0$ and $(nw + z^n)dz-zdw=0$, $n\in\mathbb{N}$, and they are related by flipping of the invariant fiber. 
	
	
	For showing that there exists infinitely  many sections in  $\mathbb{CP}^1$-bundles, it follows immediately of next result:
	
	\begin{thm}\textbf{(\cite{LM})} The composition of a finite number of flippings in a trivial bundle $S\times \mathbb{CP}^1$ gives a $\mathbb{CP}^1$-bundle over a compact Riemann surface $S$ and every $\mathbb{CP}^1$-bundle over   $S$ can be obtained of this way.
	\end{thm}
	
	So we can take the images of infinitely many sections of the trivial bundle over $S$ by the composition of flippings assured by the above theorem and therefore we have at least a non-invariant section between them.
		
\subsection{Projective structures with prescribed monodromy}		
	
To prove the Theorem \ref{teoremaexistencia}, we calculate the projective charts by projecting the section along of the leaves in a fiber transversal to the foliation. This construction was already known for the case without branch points and it can be extended in a similar way when there are branch points. It can be calculated through the image of  $\sigma(S)$ by local first integral $h$ composing with the inverse of local first integral restricted to a transversal fiber $F_1$. Therefore, the local submersions that define the regular foliation restricted to the curve $\sigma(S)$ define charts of branched projective structure on $S$. The tangency points between  $\sigma(S)$ and the foliation produce the critical points of the charts.

A priori, we do not have this control at the cusps with singularities of Fuchsian-type. We can perform this, if the surface is finite-type, after the compactification of the suspension $S^* \times_\rho \mathbb{CP}^1$, the foliation becomes a singular foliation. We can only use the construction above when the section does not pass through the singular points of the foliation, because in these points there exist no local submersion. However, we can extend the construction to the points where there is a closed meromorphic 1-form that defines the  foliation locally, such as, for example, the form $\omega=\frac{dz}{z} + \lambda \frac{dw}{w}$ is a closed meromorphic and has a Liouvillian first integral $h(z,w)=zw^\lambda$.
Therefore, around a singular point the foliation $\mathcal{F}_\rho$ comes in a closed meromorphic form that has a Liouvillian first integral, and the projective chart can be calculated in the same way above.

Finally, the developing map is the local inverse of the holonomy germ  $f$ between the transversal fiber $F_1$ and the section.



In the next proposition, we will prove that the choice of the local model in the compactification, up to birational isomorphism, is related to the section of the $\mathbb{CP}^1$-bundle passes through the singular points of the foliation or not.
	
	\begin{Proposition}
		Let $\mathcal{F}$ be a Riccati foliation in $ {\mathbb {D}} \times \mathbb{CP}^1 $ and a non-trivial section $ \sigma $  and non-invariant by foliation that intersects a singularity. Then,
		\begin{myenumerate}
		
		\item The tangency order of the image of $\sigma$ by a flipping decreases 1  with the initial section with foliation, the Liouvillian first integral also changes, but the coordinate chart around the singularity of projective structure does not change.
		\item There is one only model, up to flippings, where the section does not intersect the singularities of foliation.
	\end{myenumerate}
\end{Proposition}
\begin{proof}
Suppose non-parabolic case with model $\alpha wdz-zdw=0$, $\alpha\in\mathbb{C}$. We consider the case where the section passes through one of the singular points of the foliation $\sigma:\mathbb{D}\longrightarrow\mathbb{D}\times\mathbb{CP}^1$ given by $\sigma(z)=(z,\sigma_1(z))=(z,z^n\phi(z))$, $n\geq 1$ and $\phi(0)\neq 0$. We have that $F=\pi^{-1}(0)$ is a invariant fiber by the foliation whose monodromy is given by $e^{2\pi i \alpha}w$. The fixed points of the monodromy, $w=0$ and $w=\infty$, represent the separatrices of the foliation which pass through singular points $(0,0)$ and $(0,\infty)$.

Let $h(z,w)=z^\alpha w^{-1}$ be a multi-valued first integral of the foliation in a neighborhood $U$ of $(0,\sigma_1(0))$. We remark that the (multi-valued) graphs of $w=cz^\alpha$, $c\in\mathbb{C}$, are the leaves. Let us study the projection of $\sigma\cap U$ along the leaves in transversal fiber to a foliation $F_1=\pi^{-1}(z_1)$, $z_1\in\mathbb{D}^*$.

If $\sigma_1(0)=0$, it follows that $f(z)=z_1^{\alpha}\cdot \frac{\sigma_1(z)}{z^\alpha}$. As $\sigma_1(z)$ is a holomorphic germ with $\sigma_1(0)=0$, we can rewrite as $\sigma_1(z)=z^n\phi(z)$, $n\geq1$ and $\phi(0)\neq0$, up to an automorphism of $\mathbb{CP}^1$, $f(z)=(z\tilde{\phi}(z))^{n-\alpha}=(k(z))^{n-\alpha}$ where $\tilde{\phi}$ is the only holomorphic solution in the neighborhood of $0$ of the equation $\tilde{\phi}(z)^{n-\alpha}=\phi(z)$, $\phi(0)\neq 0$, and $k$ is a invertible germ. Thus, $f(k^{-1}(z))=z^{n-\alpha}$.

If $\sigma_1(0)=\infty$, 
 in the analogous way, we obtain the coordinate chart $z^{n+\alpha}$.

For the parabolic case, the foliation is induced by $\omega=dz-zdw$. We have to $F=\pi^{-1}(0)$ is the invariant fiber through the foliation whose monodromy is given by $w+1$. It has a saddle-node in the point $(0,\infty)$ whose weak separatrix is $F$ and the strong separatrix is transversal to $F$.

Let $h(z,w)=\log  z-w $ be the holomorphic first integral of the foliation in a neighborhood $U$ of $(0,\sigma_1(0))$. At the same way as above, we obtain  $f(z)=\log z_1-\log  z+\sigma_1(z)$.

We affirm that there exists diffeomorphism germ $w$ such that $\log\left(w(z)e^{\frac{1}{w(z)^n}}\right)=\log\left(ze^{-\sigma_1\left({z}\right)}\right)$ where $n$ satisfies $\sigma_1\left({z}\right)=\frac{1}{z^n}\cdot\sigma_2(z)$, $\sigma_2(0)\neq 0$. 

In fact, we put $w(z)=zh(z)$
$\log  h(z)+\frac{1}{z^nh^n}=-\frac{\sigma_2(z)}{z^n}.$ We take $F(z,\zeta)=z^n\log \zeta +\frac{1}{\zeta^n}+\sigma_2(z)$, where $F(z,h(z))\equiv0$. Since $\frac{\partial F}{\partial \zeta}(0,h(0))=-\frac{n}{h(0)^{n+1}}\neq 0$, since $\frac{1}{h(0)^n}=-\sigma_2(0)\neq0$, therefore $h(0)\neq 0$. 

By Implicit Function Theorem, there will exist $h(z)$ holomorphic in the neighborhood of $0$, thus $w(z)$ is a invertible germ ($w'(0)=h(0)\neq0$).	Thus, $dev(w^{-1}(z))=\log z+\frac{1}{z^n}$.

\begin{myenumerate}
	\item After one blowing up ($z=z$ e $w=zy$) and one contraction of the fiber $\{z=0\}$, we obtain a section $\beta(z)=(z,z^{n-1}\phi(z))$, thus decreases 1 in the tangency order of the initial section with the foliation. Already foliation after the flipping becomes $(\alpha-1)ydz-zdy=0$ with first integral $z^{\alpha-1}y^{-1}$ different of the initial. But we obtain the same projective chart equal to $z^{n-\alpha}$.

	\item The flipping given by the composition of $n$ blowings up and contractions sends the section $\sigma$ in the section $\beta(z)=(z,\phi(z))$, $\phi$ biholomorphism germ, $\phi(z)\neq0$, that is, a transversal section to the foliation and follows of i) that it's the only flipping that happens this.
\end{myenumerate}
We can use the same idea for the parabolic case.
\end{proof}

Therefore, we can choose a local model to compactify the suspension over $S^*$ such that the section is transversal to the foliation around  invariant fibers, i.e, it doesn't intersect the foliation's singularities.

	\section{Proof of the Existence Theorem}

	\begin{proof}[Proof of the Theorem \ref{teoremaexistencia}]
		At the $\mathbb{CP}^1$-bundle $ \pi: \overline {S^* \times_ \rho \mathbb{CP}^1} \rightarrow S$, by the discussion in the previous section $\pi$ has at least one non-trivial, and non-invariant holomorphic section $\sigma$.
		We will study two cases: regular points and cusps.
		
		\textbf{1st case:} Regular points.
		
		At regular points on the surface, we will obtain, up to appropriate coordinates changing, complex projective charts or branched coverings.
		
		In fact, at a regular point $p=(z_0,w_0)\in S^*\times\mathbb{CP}^1$ of a non-trivial and non-invariant section $\sigma$ we have that $\mathcal{F}_\rho \ \tpitchfork\ \pi$. We analyze two cases: whether $\sigma$ is transversal to  $\mathcal{F}_\rho$ at $p$ or not.
		
		By introducing coordinates $(z,w)\in S^*\times\mathbb{CP}^1$ centered on $p$, in a neighborhood of $p$ to foliation $\mathcal{F}_\rho$ is regular, we can think as "horizontal" foliation $\frac{\partial}{\partial z}$.
		
		Putting $\sigma(z):=(z,\sigma_1(z))$, let $U$ be a neighborhood of $(0,\sigma_1(0))$ in $\mathbb{D}\times\mathbb{CP}^1$, $h(z,w)=z$ the holomorphic first integral of $\mathcal{F}_\rho$ in $U$ and  $F_1=\pi^{-1}(z_1)$ a fiber near to $0$. The restriction of $h$ to $F_1\cap U$ is a diffeomorphism and $\left(\restr{h}{F_1\cap U}\right)^{-1}\circ h(\sigma(z))=(z_1,f(z))$, then $f(z)=\sigma_1(z)$.
		
		In this case, $f$ is holomorphic and if the foliation is transversal to the section in $(0,\sigma_1(0))$, we obtain that projective chart around to $0$ is a homeomorphism. Otherwise, $\sigma_1'(z)=0.$ We can rewrite as $\sigma_1(z)=z^n\phi(z)$, $\phi(0)\neq0$, or better, $\sigma_1(z)=(z\tilde{\phi}(z))^n=(k(z))^n$, $n-1$ is the tangency order of the section with the foliation and since $k'(0)\neq0$, so $k$ is a invertible holomorphic germ.
		
		Thus, $f(k^{-1}(z))=z^n$, i.e., ramified covering with $n$ sheets.
		
		\textbf{2nd case:} Cusps
		
		At the points $ \{p_1, \ldots, p_k \} $ we will obtain, up to appropriate coordinates changing, singular projective charts  $ z \mapsto z^\alpha $, $ \alpha \in  \mathbb{C}^* $ , if the monodromy around the point is non-parabolic and when it's parabolic it will be$ z \mapsto \log z + \frac{1}{z^n} $, $ n \in \mathbb{N} $.
		
		The foliations used in Lemma \ref{compactidefibrados} have Liouvillian first integrals. We shall separate in parabolic, non-parabolic, and trivial cases.
		
	In coordinates $(z,w)\in \mathbb{D}\times\mathbb{CP}^1$, we consider the case where the section $\sigma:\mathbb{D}\longrightarrow\mathbb{D}\times\mathbb{CP}^1$ given by $\sigma(z)=(z,\sigma_1(z))$ doesn't pass through singularities of the foliation, i.e., $\sigma_1(0)\neq 0,\infty$.
		
		\begin{myenumerate}
			\item Non-parabolic Monodromy
			
			 The foliation is induced by $\omega=\alpha wdz-zdw$. We have that $F=\pi^{-1}(0)$ is a invariant fiber by the foliation whose monodromy is given by $e^{2\pi i \alpha}w$. The fixed points of the monodromy, $w=0$ and $w=\infty$, represent the separatrices of the foliation which pass through singular points $(0,0)$ and $(0,\infty)$.

			Let $h(z,w)=z^\alpha w^{-1}$ be a multi-valued first integral defined at $\mathbb{D}\times \mathbb{CP}^1$ of the foliation in a neighborhood $U$ of $(0,\sigma_1(0))$. We remark that the (multi-valued) graphs of $w=cz^\alpha$, $c\in\mathbb{C}$, are the leaves. Let us study the projection of $\sigma\cap U$ along the leaves in transversal fiber to a foliation $F_1=\pi^{-1}(z_1)$, $z_1\in\mathbb{D}^*$.
			
			In fact, $\left(\restr{h}{F_1}\right)^{-1}\circ (h(\sigma(z))=\left(\restr{h}{F_1}\right)^{-1}\circ \left(\frac{z^{\alpha}}{\sigma_1(z)}\right)=\left(z_1,z_1^{\alpha}\cdot\frac{\sigma_1(z)}{z^\alpha}\right)$. Thus, $f(z)=z_1^{\alpha}\cdot \frac{\sigma_1(z)}{z^\alpha}$. Since $\sigma_1(z)$ is a holomorphic germ with $\sigma_1(0)\neq 0$, up to an automorphism of $\mathbb{CP}^1$, f is $\frac{\sigma_1(z)}{z^\alpha}=(z\tilde{\phi}(z))^{-\alpha}=(k(z))^{-\alpha}$. We remark that the equation $\tilde{\phi}(z)^{-\alpha}=\sigma_1(z), \ \sigma_1(0)\neq0$, admits only one solution $\tilde{\phi}(z)=e^{-\frac{1}{\alpha}\log(\sigma_1(z))}$ holomorphic in the neighborhood of $0$ with $k$ a invertible germ, since $k'(0)=\tilde{\phi}(0)\neq 0$. Thus, $f(k^{-1}(\frac{1}{z}))=z^{\alpha}$.

			\item Parabolic Monodromy
			
			
			We have that $f(z)=\log z_1-\log  z+\sigma_1(z)$, as above, with coordinates appropriate changing, we have $dev(w^{-1}(z))=\log z$.

			\item Trivial Monodromy 
			
		The foliation is induced by $\omega=mwdz-zdw$ where $m\in\mathbb{N}$. The fiber $F=\pi^{-1}(0)$ is invariantand the monodromy around it is the identity. The singular points of the foliation are $(0,0)$ and $(0,\infty)$. 
			
			In an analogous way to the previous cases, we obtain $f(z)=z^{n-m}$, when $\sigma_1(0)\neq\infty$.
			
			
			
			
		\end{myenumerate}
	\end{proof}
	
	\begin{remark}
		The case of trivial monodromy around a cusp, the charts are as in the case of the regular points where the section is not transversal to $\mathcal{F}_\rho$. In fact, if in Lemma \ref{compactidefibrados} we choose the model $dw=0$ instead of $mwdz-zdw=0$, for some $m\in\mathbb{Z}$, the foliation (also the first integral) would extend holomorphically at the cusps with trivial monodromy.
	\end{remark}
	
Let $\rho:\pi_1(S^*)\rightarrow PSL_2(\mathbb{C})$ a representation. We obtain a dictionary between a triple $(\pi,\mathcal{F}_\rho, \sigma)$, where $\pi:P\rightarrow S$ is a $\mathbb{P}^1$-bundle equipped with a Riccati foliation $\mathcal{F}_\rho$ and $\sigma$ is a holomorphic section of $\pi$ generically transversal to $\mathcal{F}_\rho$, and a singular projective structure of Fuchsian-type in the Riemann surface $S$ with monodromy $\rho$.

\section{Isomonodromic Deformations}
In this section, we study geometry and topology of the local structures around the cusps for deform continuously projective structures on surfaces of finite-type that preserve the holonomy. 

\subsection{Geometry and topology of Fuchsian-type singularities}
To answer the problem of Gallo-Kapovich-Marden about non-uniqueness of projective structure on surfaces of finite-type, singularities of Fuchsian-type have complicated behaviors and it was studied to assist in surgeries like moving branch points.

We denote by $\mathbb{D}^*$ a deleted open neighborhood of a cusp, up to biholomorphism, and $p: T\rightarrow \mathbb{D}^*$ universal covering of $\mathbb{D}^*$ where $T=\{x\in\mathbb{C}\ |\ \Re x<0\}$ with $p(x)=e^x$, where $\Re x$ represents the real part of $x$. Let $f$ be a multi-valued function, we take  $\tilde{f}$ a lifting of $f$ to universal covering.

\begin{Definition}The degree of a multi-valued map $f\colon \mathbb{D}^*\dashrightarrow\mathbb{D}^*$ is the maximum number of preimage of each $z\in\mathbb{D}^*$ by $\tilde{f}$ restricted to a fundamental domain.
\end{Definition}

We remark that for a multi-valued map $f$ of degree $1$, $\tilde{f}$ is injective on each fundamental domain.

\begin{Example}The multi-valued map  $z^{\frac{3}{2}}$ defines a projective chart around a cusp with monodromy $w\mapsto -w$ and developing map $e^{\frac{3}{2}x}$ defined in $T$. The points of $\{z\in\mathbb{D}^*\ |\ \Im\ z<0\}$, where $\Im x$ represents the imaginary part of $x$, has a preimage in the fundamental domain $T_0=\{x\in T\ |\ 0<\Im x< 2\pi\}$, while in $\{z\in\mathbb{D}^*\ |\ \Im z>0\}$ has two preimages in $T_0$, therefore this map has degree 2.
\end{Example}

Given a local non-parabolic monodromy conjugate to $ w \mapsto e^{2\pi i \alpha} w$, when $\Re \alpha \neq 0 $. The projective structure defined around one of the fixed points of this monodromy, which we will assume to be the origin, can be thought as a sector of $\mathbb{CP}^1 \setminus \{0,\infty \}$ centered on $0$ with angle opening $2 \pi \Re \alpha$ and length sides $1$ and $ e^{-2\pi \Im \alpha}$ identified by the transformation $w \mapsto e^ {2\pi i \alpha}w $.

Geometrically, two points $u+ iv, u'+iv' \in T $ have the same image by $dev$ if and only if $ (u',v') = (u,v) - \beta \mathbb{Z} $, where $ \beta = \frac {2 \pi i} {\alpha}$.

We put $\alpha=a+ib$. The semi-plane $au-bv <0$ is decomposed into biholomorphic strips to the disk minus the radius $[0,1)$ by $dev$ and this decomposition is given by  parallel lines to $av + bu = 0$ and equidistant with distance $ \frac{2 \pi}{| \alpha |} $.

\begin{figure}[h!]
	\centering
	\includegraphics[scale=0.5]{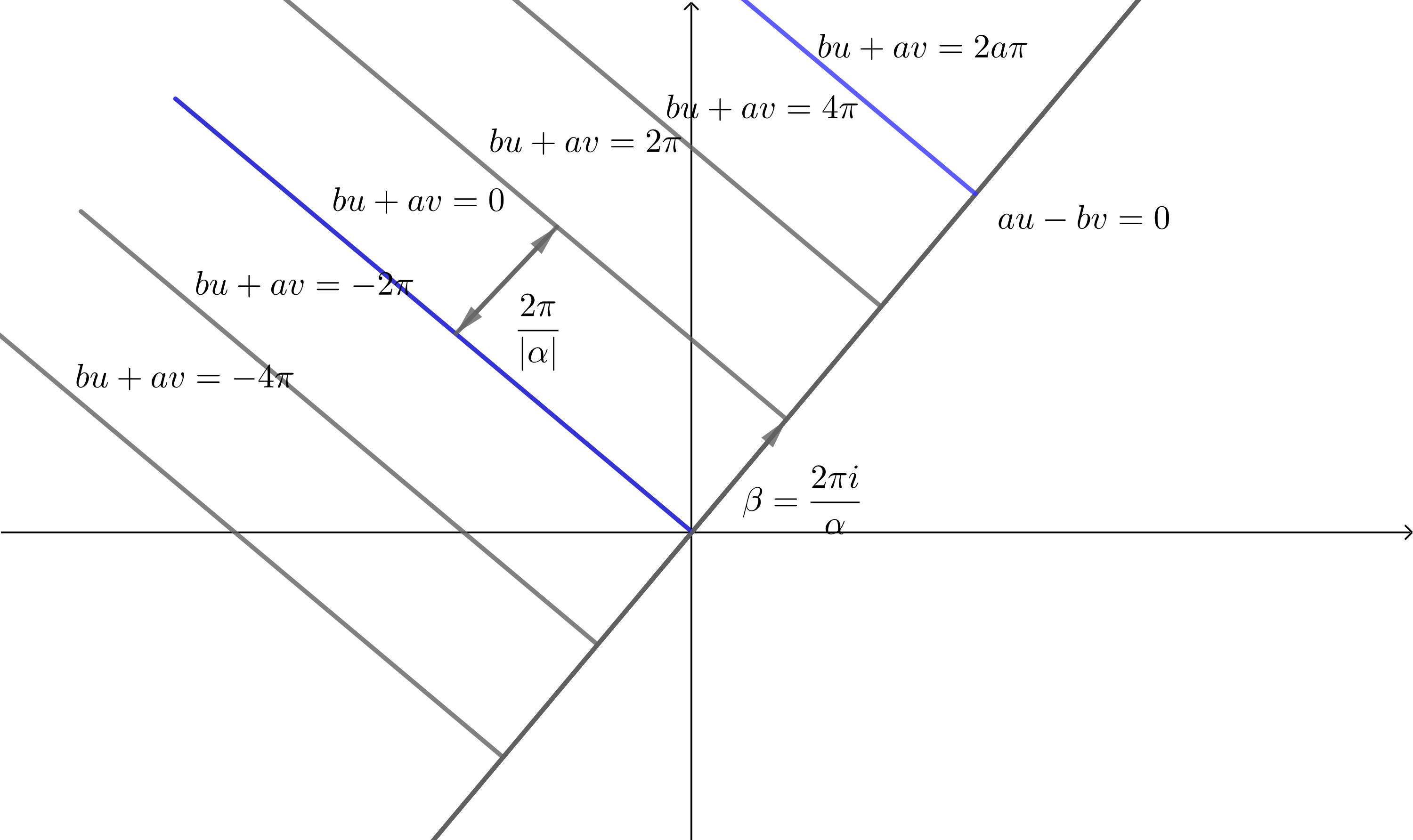}
	\caption{Decomposition of $dev(x)=e^{\alpha x}$, $\Re \alpha>0$}
	\label{geometriazalpha}
\end{figure}

We change the universal covering of $\mathbb{D} ^*$ such that the new fundamental domain is given by a strip whose boundary consists of lines $bu + av = 0$ and $ bu + av = 2 \pi|\alpha|$. Therefore, the maximum number of preimages of $z \in \mathbb{D}^*$ for $dev$ restricted to fundamental domain is $ \lceil \Re \alpha \rceil $ - degree of $ z^\alpha $.

In the case $\Re \alpha = 0$, the fundamental domain $T_0$ covers a ring $A = \{z \in \mathbb{C} \ |\ e^{- 2 \pi b} <|z| <1 \}$ through $dev(x) = e^{ibx}$. The action of $dev$ is defined by the translation $ w \mapsto w-\frac{2 \pi}{b}$ where the semi-plane $ v> 0 $ will be decomposed by the   lines parallel to $u = 0$ and equidistant with distance $\frac{2\pi} {|b|} $, see figure \ref{geometriazalpha0}. In that case, the projective structure can be seen as the ring $A$ with the boundary lines identified by the transformation $ w \mapsto e^{-2 \pi b} w$, which is topologically a torus.

\begin{figure}[h!]
	\centering
	\includegraphics[scale=.7]{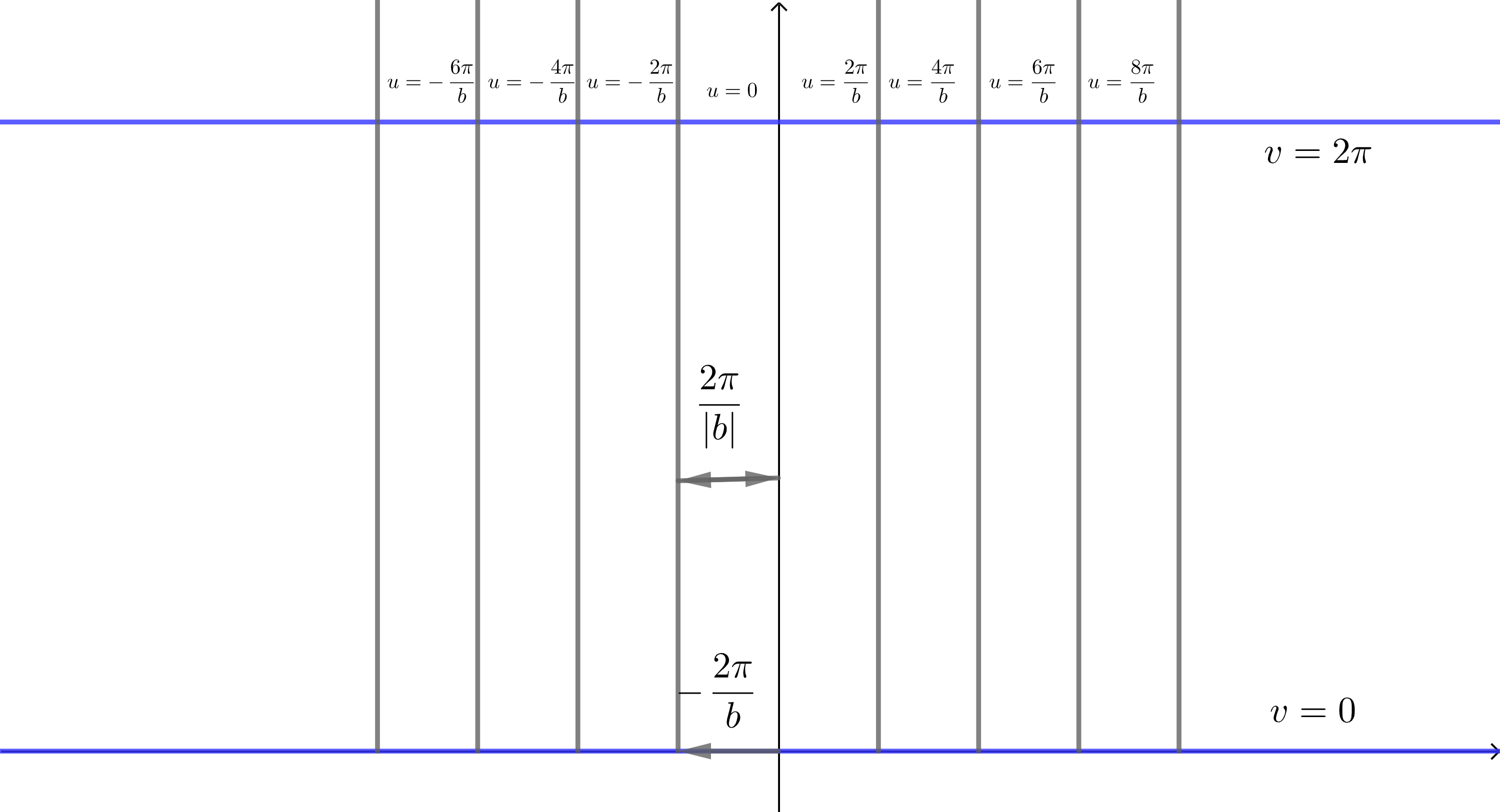}
	\caption{Decomposition of $dev(x)=e^{\alpha x}$, $\Re \alpha=0$} 
	\label{geometriazalpha0}
\end{figure}

 We prove that two actions in the universal covering of $\mathbb{D}^*$ classify projective structures of type $z^\alpha$, $\Re \alpha> 0$. This cover all charts of type $z^\alpha$, $\Re \alpha\neq 0$.

\begin{Proposition} \label{acaorecuniversal}
	The projective structure defined by the branches of $z^\alpha $, $ \Re \alpha> 0 $, in $ \mathbb {D}^*$ is represented by a pair of vectors $(2 \pi i, \frac{2\pi i} {\alpha})$, where $x \mapsto x + 2 \pi i $ and $ x \mapsto x + \frac{2 \pi i}{\alpha} $ are in $ \pi_1(\mathbb{D}^*)$ acting in T. Conversely, this pair defines the projective structure defined by branches of $z^\alpha $ in $\mathbb{D}^*$.
\end{Proposition}

\begin{proof}The first assumption follows by the discussion above. Conversely, given the pair $(2 \pi i, \frac{2 \pi i}{\alpha})$, we establish that $ 2 \pi i $ is the vector of $\pi_1(\mathbb {D}^*)$ action in T and $\frac{2 \pi i}{\alpha}$ is the vector of equivalence action of $dev$ by the monodromy representation.
	
	This pair is associated with the structure coming from the branches of $z^\alpha $, if we show that there is a biholomorphism $ \tilde{\phi} : T \rightarrow T $ such that $ \tilde{\phi} (t +2 \pi i) = \tilde{\phi}(t) +2 \pi i $ and $dev = e^{\alpha x} \circ \tilde{\phi}$ but, if we put $\tilde{\phi} = id $ and the result follows.
\end{proof}

We recall that translation structure on a surface is defined as an atlas such that the coordinate changes are translations. The branches of $\log z $ and $\log z + \frac{1}{z}$ define different translation structures in $\mathbb{D}^*$, for example. We will show that these structures and their pull-backs by covering maps of degree $\geq 2$ provide us with a list of translation structures in $\mathbb{D}^*$ modulo projective equivalence.

\begin{Proposition} \label{recobrimentolog}
	Translation structures in $\mathbb{D}^*$ induced by the \textit{pull-back} of $\log z$ by the map $z \mapsto z^n $, $n \in \mathbb{N}$ , $n \geq 2$, are projectively  equivalent to those induced by $ \log z$.
\end{Proposition}

The translation structure defined by $\log z$ can be seen as an infinite cylinder with one end.

We know that a branch of  $\log z $ is injective in its domain, the same occurs with $\log z + \frac{1}{z}$. We use some ideas from Section 2.2 of \cite{Bo} where local models of poles of meromorphic forms that induce translation structures on compact Riemann surfaces were studied.

Let $U_R=\{z\in\mathbb{C}\mid |z|>R\}$ and $V_R$ be the Riemann surface obtained after removing from $U_R$ the $\pi$-neighborhood of the real half-line $\mathbb{R}^{-}$, and identifying the lines $-i\pi+\mathbb{R}^{-}$ and $i\pi+\mathbb{R}^{-}$ by the translation  $z\mapsto z+2\pi i$.

We choose the usual determination of $\log z$ in $\mathbb{C}\setminus\mathbb{R}^{-}$ restricted to $U_{R'}$, we obtain the map $z\mapsto z+\log z$ well-defined from $U_{R'}\setminus\mathbb{R}^{-}$ to $\mathbb{C}$.

\begin{Proposition}\label{holoinjetiva}
	The map $z+\log z$ extends to a injective holomorphic map $f:U_{R'}\rightarrow V_R$, if $R'$ is large enough.
\end{Proposition}
\begin{proof}
	See \cite[Section 2.2]{Bo} or \cite[Proposição 4.3]{NS} for more details.
\end{proof}

We conclude that $ \log z + \frac{1}{z}$  is also injective when restricted to a deleted neighborhood of the origin. It remains to show that $f$ is surjective in a neighborhood of infinity and to conclude through Proposition \ref{holoinjetiva} that this chart, defined in the neighborhood of origin, is topologically $ V_R $. In fact, we have

\begin{Proposition}The map $f:U_{R'}\rightarrow V_R$ is surjective in a neighborhood of  infinity, i.e., for $Z\in V_R$ with large enough modulus, there exists $z\in U_{R'}$ such that $f(z)=Z$.
\end{Proposition}
\begin{proof}
		See \cite[Section 2.2]{Bo} or \cite[Proposição 4.4]{NS} for more details.
\end{proof}

\begin{Proposition}\label{recobrimentologz+}
	The projective structure on $\mathbb{D}^*$ given by the \textit{pull-back} of $\log z+\frac{1}{z}$ by $z\mapsto z^n$ is projectively equivalent to $\log z+\frac{1}{z^n}$.
\end{Proposition}

Then  $\log z+\frac{1}{z^n}$, in terms of projective structure, is a suitable rotating and rescaling covering of $V_R$ of order $n$.

\subsection{Generalization of surgery}
Moving branch points is a surgery that consists of deformation of branched local projective charts; it can be thought as a configuration analogous to Schiffer's variations in Riemann surface's theory, as \cite{N}. These movements were introduced by Tan in \cite{T} for projective structures with simple branch points and then generalized in \cite{CDF} for higher- order branch points. Schiffer variations, in particular the moving branch points, produce deformations of the projective structure without changing the monodromy representation but, in general, do not preserve the underlying complex structure.

Let $S$ be a closed Riemann surface with a singular projective structure of Fuchsian-type with developing map $dev$. Let $ p $ be a singularity of Fuchsian-type. 

\begin{Definition}
	We define a pair of \textit{twins embedding} in $S$ as a pair of embedded curves $\gamma = \{\gamma_1, \gamma_2 \}$ starting from $p$ such that there is a determination of developing map around $ \gamma_1 \cup \gamma_2 $ which maps $\gamma_1 $, $\gamma_2 $ into a
	  simple curve $\hat{\gamma}\subset\mathbb {CP}^1$.
	\end{Definition}

According to the study of degree of $dev(x) =e^{\alpha x}$  done in the previous section, we have that there are twin curves for $ \Re \alpha>1$. In the case of parabolic monodromy, we saw in the previous section that $\log z$ and $\log z + \frac{1}{z} $ are injective in $\mathbb{D}^*$, so they don't have twins. For $ n \geq 2 $, $\log z + \frac{1}{z^n}$ can be seen as a branched covering of order $n$ of $\log z + \frac{1}{z} $, as it was done in  Proposition \ref{recobrimentologz+} and then the preimage of a segment will have $n$ copies in $\mathbb{D}^*$ for $\log z+\frac {1}{z^n}$ are the candidate twins in this model.

We'll describe the moving branch points as it was done in \cite{CDF}. Let $p$ be a branch point of $S$, we take twin curves $\gamma_1$, $\gamma_2 $ starting from $ p $ with endpoints $q_1$ and $ q_2 $. We denote by $ \alpha $ and $ \beta $ angles in $ p $, and  $ \theta_i $ the angles in $ q_i $, $ i = 1,2 $, where $ \theta_i = 2 \pi $ , if $ q_i $ is a regular point. A new branched projective structure in $S$ will be obtained by cutting $S$ along $ \gamma_1 \cup \gamma_2$ and pasting the copies according to the identifications made in Figure \ref{movpontoramificacao}.

\begin{figure}[h!]
	\begin{tikzpicture}
		\draw (-1,2)node{$\ast$}node[below]{$q_1$} circle(0.18)node[above left=1pt]{$\theta_1$}--(1,2)node[midway,below]{$\gamma_1$}node{$\circ$}node[below right=2pt]{$p$} circle(0.18)node[above=2pt]{$\alpha$}node[below =2pt]{$\beta$}--(3,2)node[midway,below]{$\gamma_2$}node{$\ast$}node[below]{$q_2$}circle(0.18)node[above right=1pt]{$\theta_2$};
		
		\draw[->](4,2)--(5,2)node[midway,above]{cut};
		
		\draw (7.5,0.3)node{$\circ$}node[below]{$p_1$}--++(1.8,1.8)node[midway,sloped]{$||$}node{$\ast$}node[right]{$q_2$}--++(-1.8,1.8)node[midway,sloped]{$|$}node{$\circ$}node[above]{$p_2$}--++(-1.8,-1.8)node[midway,sloped]{$|$}node{$\ast$}node[left]{$q_1$}--cycle node[midway,sloped]{$||$};
		
		\draw[->](10,2)--(11,2)node[midway,above]{paste};
		
		\draw (12,0)node{$\circ$}node[below]{$p_1$}circle(0.18)node[left=2pt]{$\beta$}--(12,2)node[midway,right]{$\gamma_2'$}node{$\ast$}node[right=2pt]{q}circle(0.18)node[below left=2pt]{$\theta_1+\theta_2$}--(12,4)node[midway,right]{$\gamma_1'$}node{$\circ$}node[above]{$p_2$}circle(0.18)node[left=2pt]{$\alpha$};
		
	\end{tikzpicture}
	\caption{Moving branch points}
	\label{movpontoramificacao}
\end{figure}
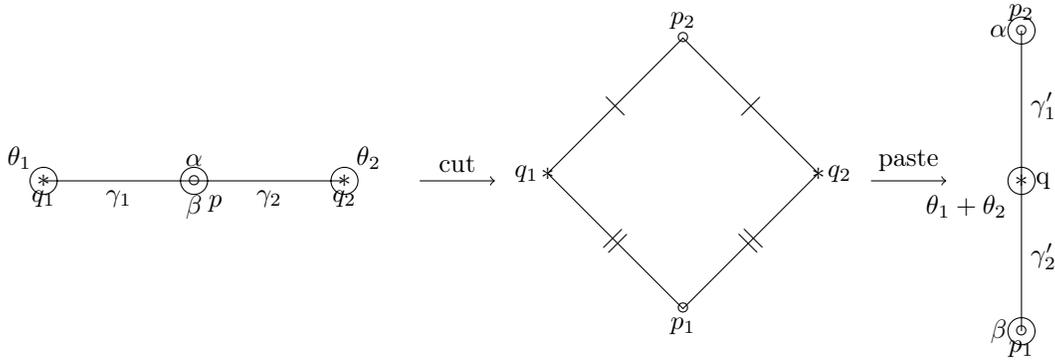

After this process, we obtain two new twin curves $ \gamma'_1 $ and $ \gamma'_2 $ starting from a point $q$ with total angle $ \theta_1 + \theta_2 $ and the endpoints $p_1$ and $p_2$ of the new twins have angles $ \beta $ and $ \alpha $, respectively. Note that the image of $ \gamma'_1 $ and $ \gamma'_2 $ by the developing map is the same as $ \gamma_1 $ and $ \gamma_2 $, that is, the surgery does not change the image by dev, then we have a pair of twins embedded into the new structure, and we will return to the initial structure if we move the points along of that pair.

Note that the angles $\alpha$, $\theta_i$ and $ \beta$ are multiples of $2 \pi $ and if $p$ is a single branch point, then $\alpha = \beta = 2\pi $. This process describes locally a continuous deformation in the space of classes of branched projective structures over $ S $ with monodromy representation fixed. In addition to being used to collapse branch points, the process can be used to change the position of branch points and to separate a higher-order branch point into several branch points of lower order.

Calsamiglia, Deroin, and Francaviglia in \cite{CDF} prove that two-branched projective structures in compact surfaces with the same quasi-Fuchsian holonomy and the same degree of branching are related by a movement of branch points, so it is possible to use this surgery to show the non-uniqueness of projective structures with the same monodromy representation.

A cone-angle $\theta$ is produced from a sector of angle $\theta$ by identification of their boundaries by an isometry. Then the singularities with cone-angle, called conical singularities, still have the same notion of angle as they have in the case of branch points. So the surgery to move branch points will work in the same way. 

In \cite{Tr}, Troyanov characterized orientable compact surfaces with conical singularities. The invariants that represent the opening of the cone are real numbers and, he obtained a classification of these surfaces. More precisely, given $ p_1, \ldots, p_k \in S$ and $\theta_1, \ldots, \theta_k> 0 $, if $\chi(S) + \sum_{i=1}^k(2 \pi- \theta_i)<0 $ (respectively, $=0$ or $=1$), then there exists a hyperbolic metric (respectively, Euclidean or spherical) in $S \setminus \{p_1, \ldots, p_k \}$ with a conical angle $\theta_i$ in $p_i$.

Now, we will prove the  inverse surgery for the case $\Re \alpha>1$, that is, when the degree of multi-valued function $z^\alpha$ is at least 2. Given two twin curves starting of $z^\alpha$, we remove an angle $2\pi$, that we see in a fundamental domain, it would be to remove one of the biholomorphic strips of the disk minus a radius through of $dev(x)=e^{\alpha x}$ and glue in the perpendicular way to the boundary of strip. 

\begin{Proposition}\label{surgeryinverse}Let $\gamma_1$, $\gamma_2$ be a pair of twins that start from a singularity $p$ of type $z^\alpha$, with $\Re \alpha>1$, forming a sector with angle $2\pi$ and end-points $q_1$ and $q_2$ are regular points. The  inverse surgery of the movement that removes the angle $2\pi$ in $p$ results in a simple branch point where start two twin curves whose end-points are a singularity of type $z^{\alpha-1}$ and a regular point.
\end{Proposition}

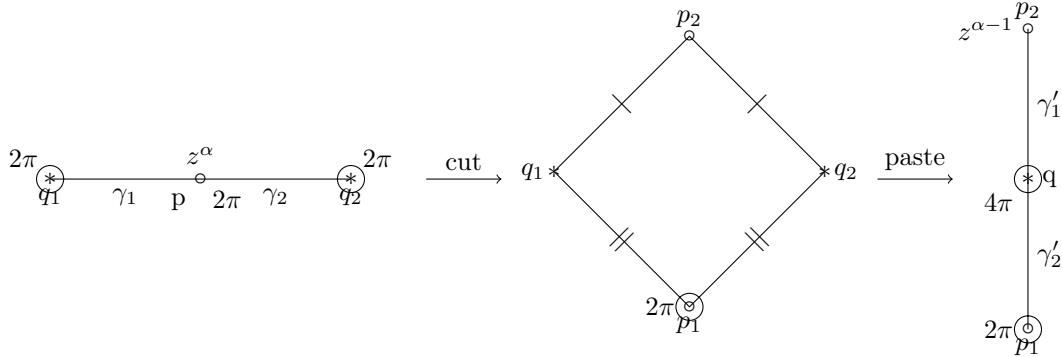
\begin{figure}[h!]
	\centering
	\begin{tikzpicture}
		\draw (-1,2)node{$\ast$}node[below]{$q_1$} circle(0.18)node[above left=1pt]{$2\pi$}--(1,2)node[midway,below]{$\gamma_1$}node{$\circ$}node[below left=2pt]{p} node[above=2pt]{$z^\alpha$}node[below right=1pt]{$2\pi$}--(3,2)node[midway,below]{$\gamma_2$}node{$\ast$}node[below]{$q_2$}circle(0.18)node[above right=1pt]{$2\pi$};
		
		\draw[->](4,2)--(5,2)node[midway,above]{cut};
		
		\draw (7.5,0.3)node{$\circ$}node[below]{$p_1$}circle(0.18)node[left=2pt]{$2\pi$}--++(1.8,1.8)node[midway,sloped]{$||$}node{$\ast$}node[right]{$q_2$}--++(-1.8,1.8)node[midway,sloped]{$|$}node{$\circ$}node[above]{$p_2$}--++(-1.8,-1.8)node[midway,sloped]{$|$}node{$\ast$}node[left]{$q_1$}--cycle node[midway,sloped]{$||$};
		
		\draw[->](10,2)--(11,2)node[midway,above]{paste};
		
		\draw (12,0)node{$\circ$}node[below]{$p_1$}circle(0.18)node[left=2pt]{$2\pi$}--(12,2)node[midway,right]{$\gamma_2'$}node{$\ast$}node[right=2pt]{q}circle(0.18)node[below left=2pt]{$4\pi$}--(12,4)node[midway,right]{$\gamma_1'$}node{$\circ$}node[above]{$p_2$}node[left=2pt]{$z^{\alpha-1}$};
		
	\end{tikzpicture}
	\caption{Inverse moving branch points in $z^\alpha$}
	\label{movpontoramificacaoalpha}
\end{figure}

\begin{proof} After the cut and paste process, we identify two regular points $q_1$ and $q_2$, making a simple branch point $q$. We still need to show that removing an angle $2\pi$ of $z^\alpha$ will result in a singularity of type $z^{\alpha-1}$.
	
We use the decomposition of the universal covering of $\mathbb{D}^*$ as in Figure \ref{geometriazalpha} and we note the degree of $dev $ is $\lceil\Re\alpha\rceil$.

	
We will remove a biholomorphic strip from the disk minus the radius by $dev$ and all its copies via the action of fundamental group $\pi_1(\mathbb{D}^*)$ and define a relation in the lines $av + bu = 2 \pi l$, $l \in \mathbb{Z}$, given by $u + iv \sim u + iv + j \beta $, $j \in \mathbb{Z}$. This  identifies the boundaries of strips in the direction of vector $\beta=\frac{2 \pi i}{\alpha}$.
	
	
	The initial $dev$ is given by $ D(x) = e^{\alpha x}$. Since the family of lines are twins of the projective structure in $\mathbb{D}^*$ we see that $D(0)=D(\beta)$ and it follows from the equivalence of $D$ by the monodromy representation  $\rho:\pi_1(\mathbb{D}^*)\rightarrow PSL_2(\mathbb{C})$ given by $\rho ([\gamma])=e^{2 \pi i \alpha}w$ that $D(2\pi i) = D(0 + 2 \pi i) = D (\beta)\cdot e^{2 \pi i \alpha} $, here we use the action $ x \mapsto x + 2\pi i $ of fundamental group in $ T $.
	
	Note that $D(\beta + w)=D(\beta)\cdot e^{2 \pi i \alpha}$, where $w=2\pi i-\beta$. We affirm that $D$ is equivalent for the monodromy representation $\rho$ and the new action of the fundamental group is given by $ x \mapsto x + w $. We just need to show that, $ \forall x \in T $, $D(x + w) = D(x) \cdot e^{2 \pi i \alpha}$. In fact, $D(x + w)=D(x + 2\pi i- \beta) = D(x- \beta) \cdot e^{2\pi i \alpha} = D(x) \cdot e^{2 \pi i\alpha}$, since $D(x-\beta)=D(x)$.
	
	We will obtain a new domain $T'$ of the covering application and a new developing map $D_1$ equivariant with respect to $\rho$ with the new action of the fundamental group and the images will coincide with the initial developing map at the respective paste points.
	
	The domain $T'$ is simply connected, its quotient by the action of $x \mapsto x + w$ is homeomorphic to $\mathbb{D}^*$ and therefore can be taken as a universal cover of $\mathbb{D}^* $. 
	
	Using the classification obtained in Proposition \ref{acaorecuniversal}, we have that after the surgery, they are given by $ (w, \beta) = \left(2\pi i \left (\frac{\alpha-1}{\alpha} \right), \frac{2\pi i}{\alpha} \right) $ and the linear transformation $L:\mathbb{R}^2 \rightarrow \mathbb{R}^2$ given by $L(t) = \frac{\alpha}{\alpha-1}t $ takes them to a pair of the form $(2 \pi i, \beta')$ where $\beta'=\frac{2 \pi i}{\alpha -1}$ and therefore the structure obtained is equivalent  biholomorphically to $z^{\alpha-1} $ in $\mathbb{D}^*$ with new developing map is $D_1 \circ L^{-1}$.
	
\end{proof}

We can conclude that the rigid models, that is, those that do not have twins path (do not have excess angle) and the models $ \log z + \frac{1}{z^{n}}$, $ n \geq 2 $ and $ z^{\alpha}$, $ \Re \alpha> 1 $ have twins and candidates to be deformed isomonodromically.

\section{Branching order}
In this section, we explore the Problem \ref{minimizing} posed by Gallo-Kapovich-Marden about minimizing angles on projective structures. Let  $\rho:\pi_1(S^*) \rightarrow PSL_2(\mathbb{C})$ be a representation, what is the minimum branching order of a projective structure of Fuchsian-type with this monodromy representation? 

We obtain a result with a existence of obstruction for prescribing local models without branch points and angle excess in the cusps of a compact Riemann surface.

\subsection{Algebro-geometric interpretations of projective structures of Fuchsian-type}

There are exactly two oriented topologically $\mathbb{S}^2$-bundle over the closed Riemann surface $S$ and they are distinguished by the 2nd Stiefel-Whitney class $w_2(P)$ of the bundle  $\pi:P\rightarrow S$, $\sigma^2\equiv w_2(P)(mod\ 2)$, where $\sigma$ is section of $\pi$. Then, the parity of the self-intersection $\sigma^2$ depends only on the bundle: $\sigma^2$ is even if the bundle is diffeomorphic to the trivial bundle and it is odd, otherwise.

\begin{Proposition}\label{paridadesigma} Let $\sigma$ and $\sigma'$ be two holomorphic sections of holomorphic  $\mathbb{CP}^1$-bundles on a compact Riemann surface that have the same 2nd class of Stiefel-Whitney. We have to
	$$\sigma^2 \equiv \sigma'^2 (mod \ 2). $$
\end{Proposition}
In particular, holomorphic sections of the same $\mathbb{CP}^1$-bundle have self-intersection with the same parity.

Let $ \pi: P \rightarrow S $ be a $\mathbb{CP}^1$-bundle with a Riccati foliation, after a flipping of an invariant fiber we get another $\mathbb{CP}^1$-bundle $ \pi': P' \rightarrow S $ also equipped with an  equivalent birationally Riccati foliation equivalent to $\pi$. In fact, flipping changes the topological class of the bundle:

\begin{Proposition} \label{paridadestiefelwhitney}
	The second Stiefel-Whitney classes $w_2(P)$ and $w_2(P')$ have different parities.
\end{Proposition}
\begin{proof}Let $\sigma$ be holomorphic section of fiber bundle $\pi: P \rightarrow S $, after a flipping on an invariant fiber, if we consider a blow-up at a point outside the section, after flipping we have the new section $ \tilde{\sigma}$ de $ \pi': P' \rightarrow S$ has self-intersection $ \tilde{\sigma}^2 = \sigma^2+1$, and if we consider blow-up at a point in the section, after flipping we have that the new section has self-intersection $ \sigma^2-1 $. 
\end{proof}

In general, the intersection numbers of holomorphic sections are either all even, or all odd: $\sigma^2\mod 2$ is the topological invariant of the bundle. Then, at the same compactification, two holomorphic sections have the same parity of tangency order with the foliation.



Let $ \pi: P \rightarrow S $ be a $ \mathbb{CP}^1$-bundle over $S$ associated to the monodromy representation $ \rho: \pi_1 (S^*) \rightarrow PSL_2 (\mathbb{C})$ of a projective structure of Fuchsian-type in $S$. This $\mathbb{CP}^1$-bundle is equipped with a Riccati foliation $ \mathcal{F}_\rho$ (see Brunella \cite{B}[Section 4.1]) with the same monodromy obtained by compactification the suspension of the representation $\rho$. The developing map of the projective structure defined in $ S $ defines a non-trivial holomorphic section $ \sigma$ of $ \pi$ non-invariant by $\mathcal{F}_\rho$.

We obtain a formula that relates topological invariants of the surface with the tangency order of Riccati foliation with the holomorphic section (see \cite[p. 22]{B}) of the suspension and its self-intersection.

\begin{Proposition}\label{autointersection}Under the conditions above, the self-intersection of $\sigma(S)$ in $P$ is
	$$\sigma(S) \cdot \sigma(S) = tang(\mathcal{F}_\rho, \sigma(S)) + \chi(S)-k_0, $$
	where $ k_0 $ represents the number of fibers invariant by foliation $\mathcal{F}_{\rho} $.
\end{Proposition}

\begin{proof}It's a consequence of Brunella's formula \cite{B}: the cotangent bundle of a Riccati foliation is
	$$T_{\mathcal{F_\rho}}^*=\pi^*(K_S)\otimes\mathcal{O}_P\left(\sum_{j=1}^{n}k_jF_j\right),$$
	where $K_S$ is the canonical bundle and $F_1,\ldots,F_n$ are the $\mathcal{F_\rho}$-invariant fibres of multiplicity $k_1,\ldots,k_n.$
Since after the compactification, there are $k_0$ $\mathcal{F_\rho}$-invariant fibers with multiplicity 1, we have  $T_{\mathcal{F}_\rho}.\sigma=2-2g-k_0$	with the formula $T_\mathcal{F_\rho}\cdot \sigma=\sigma\cdot \sigma-tang(\mathcal{F}_\rho,\sigma)$  (see \cite[Proposition 2.2]{B}), the result follows.
	
\end{proof}

\begin{remark}We can prove this Proposition with the same ideas of Proposition 11.2.2 of \cite{GKM}, for complete proof see \cite[Teorema 5.3]{NS}.
\end{remark}




\begin{Corollary}Let $\sigma$ and $ \sigma'$ be two  non-trivial holomorphic sections of compactified bundle over $S$ associated to monodromy representation $ \rho: \pi_1 (S^*) \rightarrow PSL_2(\mathbb{C})$. Then,
	$$ tang(\mathcal{F}_\rho, \sigma(S)) \equiv tang (\mathcal{F}_\rho,\sigma'(S)) \mod 2, $$
	where $\mathcal{F}_\rho$ is Riccati foliation of  compactified bundle.
\end{Corollary}
\begin{proof}It follows immediately from the propositions \ref{paridadesigma} and \ref{autointersection}.
\end{proof}


\subsection{Minimum branching order}

Let $\rho: \pi_1(S^*) \rightarrow PSL_2(\mathbb{C})$ be a representation.  Minimizing the branching order of a projective structure of Fuchsian-type on $S$ with monodromy $\rho$ is equivalent to minimizing the index $tang(\mathcal{F}_\rho, \sigma(S))$ of a Riccati foliation and a section $\sigma$ of the $\mathbb{CP}^1$-bundle with a specific compactification that defines a projective structure.

We recall that a branched projective structure induces a complex structure and thus angles on S. Unbranched points are called regular and the total angle around them is $2\pi$. The cone-angle around a point p whose branching order is $n_p\geq 2$ is $2\pi n_p$. The branching divisor of $\sigma$ is the divisor
$\sum_{p\in S}(n_p-1)p$. Its degree $\sum_{p\in S}(n_p-1)$ is called the total branching order of $\sigma$.

We extend the notion of branching order to singular points of Fuchsian-type. In fact, it will follow from Theorem \ref{teoremaexistencia} that around each singular point $ p $ of the projective structure $\sigma$ with given monodromy $ \rho $ the projective charts are defined by $z^{\alpha + n_p} $, $ 0 <\Re \alpha \leq1 $ or $ \log z + \frac{1}{z^{n_p}} $. We define $ n_p \in \mathbb{Z} $ as the branching order  at each singular point $p$ and the sum $e(\sigma) = \sum_{p \in S} n_p $ as the branching order of projective structure $\sigma $. We also define 
\begin{equation*}
d(\rho) = \min \{e(\sigma): \sigma \ \text{is a projective structure of Fuchsian-type with monodromy} \rho \}.
\end{equation*}
Gallo-Kapovich-Marden proved that $d(\rho)=0$ for all liftable non-elementary representations $\rho$ and $d(\rho)=1$ for all non-liftable non-elementary representations $\rho$. 

We can see the sum $e(\sigma)$ as a tangency order of a Riccati foliation with sections of fiber bundles from compactification of suspension of a representation $\rho$. We fix a complex structure on $S$, it follows from the proof of the Existence Theorem that $n_p$ are tangency orders of foliation with the section: 
$$e(\sigma)= tang(\mathcal{F}_\rho^{min},\sigma(S)),$$
where $\mathcal{F}_\rho^{min}$ is the foliation provided the compactification which the local models are:
\begin{itemize}
	\item $\alpha wdz-zdw=0$, $\alpha\in\mathbb{C}^*$ e $0\leq\Re\alpha<1$, at the cusps with non-parabolic monodromy; 
	\item $zdw-dz=0$ at the cusps with parabolic monodromy;
	\item $dw=0$ at the cusps with trivial monodromy.
\end{itemize}

In that compactification, $e(\sigma)=tang(\mathcal{F}_\rho^{min},\sigma)=0$ if and only if the section $\sigma$ is transversal to $\mathcal{F}_\rho^{min}$. For this reason, we'll call it minimum compactification.

The Theorem \ref{structcomexcessos} is about representations that are not realized as monodromy of projective structures of Fuchsian-type with minimal branching order. For the these cases, we will necessarily have $ d(\rho) \geq 1 $, that is, these cases do not realize projective structures without angle excesses.


First, it follows from the formula $\sigma^2\equiv w_2(P)(mod\ 2)$ and Proposition \ref{autointersection}.

\begin{Lemma}\label{stiefelwhitneyexcessos}
	Let $\rho:\pi_1(S^*)\rightarrow PSL_2(\mathbb{C})$ be a monodromy representation of a projective structure of Fuchsian-type $\sigma$ on $S$, we have:
	$$w_2(P)+k_0\equiv e(\sigma)\mod 2$$
	where $w_2(P)$ is 2nd Stiefel-Whitney class of the minimum compactification $\mathcal{F}_\rho^{min}$ and $k_0$ represents the number of points with non-trivial local monodromy.
\end{Lemma}

Thus, the sum of angle excesses $e(\sigma)$ has the same parity as $w_2(P) + k_0$. Since minimum compactification only depends on the monodromy, so does its 2nd Stiefel-Whitney class, and $k_0$ represents the number of cusps with non-trivial local monodromy. We conclude that the parity of the sum of angle excesses only depends on the monodromy.

If a representation $ \rho $ is the monodromy of a projective structure without angle excess, we have that the 2nd Stiefel-Whitney class of minimum compactification has the same parity as the number of invariant fibers by foliation.

Let $\rho:\pi_1(S^*)\rightarrow PSL_2(\mathbb{C})$ be a representation of fundamental group of surface of finite-type $S^*=S\setminus\{p_1,\ldots,p_k\}$, where $S$ is a closed surface of genus $g\geq1$, we consider a presentation of $\pi_1(S^*)$:
$$\langle a_i,b_i,c_j, \ i=1,\ldots, g, \ j=1,\ldots, k\ |\ \prod_{i=1}^g[a_i,b_i]\prod_{j=1}^kc_j=Id\rangle,$$
where $[a_i,b_i]=a_ib_ia_i^{-1}b_i^{-1}$ is the  commutator of $a_i$ and $b_i$, in this presentation we can define the representation $\rho$ as $\rho(a_i)=A_i$, $\rho(b_i)=B_i$ and $\rho(c_j)=C_j$, where $A_i,B_i$ and $C_j$ are elements of $PSL_2(\mathbb{C})$ that satisfy
$$  \prod_{i=1}^g[A_i,B_i]\prod_{j=1}^kC_j=Id.$$ 
For each generator $a_i,b_i$ and $c_j$, $\rho$ can lift in two ways,
$$\pm\tilde{A}_i,\pm\tilde{B}_i \ \mbox{e}\ \pm\tilde{C}_j\in  SL_2(\mathbb{C}),$$ 
whose projetivizations give the Möbius transformations of $A_i, B_i$ and $C_j$, respectively, we choose a sign for each element and the product
\begin{equation}\label{produtosl2c}
	\prod_{i=1}^g[\tilde{A}_i,\tilde{B}_i]\prod_{j=1}^k\tilde{C}_j
\end{equation} 
can be $\pm Id$. For the choices where the product gives $Id$ the representation lifts to  $SL_2(\mathbb{C})$, if it gives $-Id$, the representation does not lift to $SL_2(\mathbb{C})$.

In the case of genus $0$, the presentation of $\pi_1(S^*)$ there are not $a_i, b_i$, only $c_j$, $j=1,\ldots,k$, satisfying $\prod_{j=1}^kc_j=Id$.

\begin{remark}
	Every representation $\rho: \pi_1(S^*)\rightarrow PSL_2(\mathbb{C})$ lifts to $SL_2(\mathbb{C})$, because the group $\pi_1(S^*)$ is free. The question here is the representation lifts if we prescribe the local models (e. g. minimal angles) at the cusps or not. Namely, if we have a minimal angle at a point, then we choose a lift. 
\end{remark}

\begin{Proposition}The parity of the 2nd Stiefel-Whitney class of the minimum compactification changes depending on whether the representation lifts to $SL_2(\mathbb{C})$ or not.
\end{Proposition}
\begin{proof} Let $ \pi: P \rightarrow S $ be a $ \mathbb{CP}^1 $-bundle with minimum compactification, after a flipping of an invariant fiber we get another $\mathbb{CP}^1$-bundle $ \pi': P' \rightarrow S $ with a birationally equivalent Riccati foliation. It follows from Proposition \ref{paridadestiefelwhitney} that flipping changes the topological class of bundles and therefore the 2nd Stiefel-Whitney classes $w_2(P)$ and $w_2(P')$ have distinct parities.
	
	If $ \rho $ lifts to $ SL_2(\mathbb{C})$, we have that each generator of $\pi_1(S^*)$ lifts to a matrix in $SL_2(\mathbb{C})$, where the product given by the equation (\ref{produtosl2c}) of these matrices is $Id$. The matrices related to the local monodromy representations come from linear differential equations with simple poles used to projectivize and thus obtain the local Riccati model. When we do one flipping the sign of that matrix will change, changing the compactification and therefore the product of all matrices is $-Id$ in that compactification the representation $\rho$ does not lift.
	
	We will obtain two families that alternate parity when making a flipping: in one of the families, the parity is even and is odd in the others.
\end{proof}

We obtain a version analogous to Theorem 3.10 of Goldman's thesis \cite{GoT}:

\begin{Proposition} \label{liftastiefelwhitney}
	At the minimum compactification, the representation $\rho:\pi_1(S^*) \rightarrow PSL_2(\mathbb{C})$ lifts to $ SL_2(\mathbb{C})$ if and only if $w_2(P)$ is even. 
\end{Proposition}

\begin{proof}Build a path in the character variety $$\displaystyle Hom (\pi_1(S^*), PSL_2(\mathbb {C})) \slash PSL_2 (\mathbb{C})$$ between one representation $ \rho $ which lifts and the trivial representation, that also lifts, preserving the lifting relation expressed in the equation (\ref{produtosl2c}) equal to $Id$.
	
Using the continuity of 2nd class Stiefel-Whitney, we can deduce that it will be constant along the path, and therefore equals to zero.
\end{proof}

\begin{Corollary}At the minimum compactification, the representation $\rho: \pi_1 (S^*) \rightarrow PSL_2 (\mathbb{C})$ lifts to $SL_2 (\mathbb{C})$ if and only if $ e(\sigma) \equiv k_0  \mod 2$. 
\end{Corollary}

\begin{proof} If the representation $ \rho $ lifts, it follows from the Proposition \ref{liftastiefelwhitney} that $ w_2(P)$ is even and using the Lemma \ref{stiefelwhitneyexcessos} will follow that $e(\sigma) \equiv k_0 \mod 2 $. Similarly, $e(\sigma)$ is shown to have parity other than $ k_0 $ when the representation $\rho$ doesn't lift.
\end{proof}

\begin{proof}[Proof of Theorem \ref{structcomexcessos}] Suppose $tang(\mathcal{F}_\rho^{min}, \sigma(S))$ is odd. If the $ w_2(P)$ is even, then $\sigma^2 \equiv w_2 (P) \equiv 0 \mod 2 $ and therefore it follows from the Proposition \ref{autointersection} that $ tang (\mathcal{F}_\rho^{min}, \sigma(S)) $ $ \equiv k_0 \mod 2$, so $k_0$ is odd. Similarly, if $ w_2(P)$ is odd, it follows that $k_0$ is even.
	
The other implication follows immediately from Lemma \ref{stiefelwhitneyexcessos} above.
\end{proof}

\end{document}